\documentclass[12pt,a4paper]{amsart}
\usepackage{txfonts}

\usepackage{hyperref}
\usepackage{latexsym}
\usepackage{amssymb}

\newtheorem{theorem}{Theorem}[section]
\newtheorem{lemma}[theorem]{Lemma}
\newtheorem{corollary}[theorem]{Corollary}
\newtheorem{proposition}[theorem]{Proposition}
\newtheorem{example}[theorem]{Example}
\newtheorem{problem}[theorem]{Problem}

\theoremstyle{definition}
\newtheorem{definition}[theorem]{Definition}

\newtheorem{remark}[theorem]{Remark}

\newcommand{\eps}{\varepsilon}

\numberwithin{equation}{section}

\begin{document}

\title[Universal left-stability of $\varepsilon$-isometries]{{\bf On universal left-stability of $\epsilon$-isometries}}

\author{Lingxin Bao, $^\dag$  Lixin Cheng , $^\ddag$ Qingjin Cheng, Duanxu Dai}
\address{Lingxin Bao, Lixin Cheng, Qingjin Cheng, Duanxu Dai:  School of Mathematical Sciences, Xiamen University,
 Xiamen, 361005, China;}
\email{lxcheng@xmu.edu.cn; dduanxu@163.com; }
\thanks{$^\dag$ Support partially by NSFC, grant no.11071201 }
\thanks{$^\ddag$ Support partially by NSFC, grant no. }

\begin{abstract}
Let $X$, $Y$ be two real Banach spaces, and $\eps\geq0$. A map
$f:X\rightarrow Y$ is said to be a standard $\eps$-isometry if
$|\|f(x)-f(y)\|-\|x-y\||\leq\eps$ for all $x,y\in X$ and with
$f(0)=0$. We say that a pair of Banach spaces $(X,Y)$ is stable if
there exists $\gamma>0$ such that for every such $\eps$ and every
standard $\eps$-isometry $f:X\rightarrow Y$ there is a bounded
linear operator $T:L(f)\equiv\overline{{\rm span}}f(X)\rightarrow X$
such that $\|Tf(x)-x\|\leq\gamma\eps$ for all $x\in X$. $X (Y)$ is
said to be left (right)-universally stable, if $(X,Y)$ is always
stable for every $Y (X)$. In this paper, we show that if a dual
Banach space $X$ is universally-left-stable, then it is isometric to
a complemented $w^*$-closed subspace of $\ell_\infty(\Gamma)$ for
some set $\Gamma$, hence, an injective space; and that a Banach
space is universally-left-stable if and only if it is a cardinality
injective space; and universally-left-stability spaces are invariant.

\end{abstract}



\maketitle

\section{Introduction}
In this paper, we study properties of universally-left-stable Banach spaces for $\eps$-isometries, and search for
characterizations of such stable spaces.
 We first recall definitions of isometry and $\eps$-isometry.
\begin{definition} Let $X,Y$ be two Banach spaces, $\eps\geq0$, and $f:X\rightarrow Y$ be a mapping.

(1) $f$ is said to be an $\eps$-isometry if  $$|\|f(x)-f(y)\|-\|x-y\||\leq\eps\;\;{\rm for\; all}\; x,y\in X;$$

(2) In particular, if $\eps=0$, then the $0$-isometry $f$ is simply called an isometry;

(3) We say that an ($\eps$-) isometry $f$ is standard if $f(0)=0$.

\end{definition}

{\bf Isometry and linear isometry.} The study of properties of isometry between Banach spaces and its generalizations
 have continued for 80 years. The first celebrated result was due to Mazur and Ulam (\cite{ma},1932): Every surjective
 isometry between two Banach spaces is necessarily affine. But the simple example: $f:\mathbb R\rightarrow\ell^2_\infty$
 defined for $t$ by $f(t)=(t,\sin t)$ shows that it is not true if an isometry is not surjective. For nonsurjective isometry,
  Figiel \cite{figiel} showed the following remarkable theorem in 1968.

\begin{theorem} [Figiel]
Suppose $f$ is a standard isometry from a Banach $X$ to another Bnanach space $Y$. Then there is a linear operator
$T: L(f)\rightarrow X$ with $\|T\|\leq1$ such that $Tf(x)=x,$ for all $x\in X$; or equivalently, $Tf=I_X$, the identity on $X$.
\end{theorem}

In 2003, Godefroy and Kalton \cite{gode} studied the relationship between isometry and linear isometry, and showed the following deep theorem:
\begin{theorem}[Godefroy-Kalton]
Suppose that $X,Y$ are two Banach spaces.

(1) If $X$ is separable and there is an isometry $f: X\rightarrow Y$, then $Y$ contains an isometric linear copy
of $X$;

(2) If $X$ is a nonseparable weakly compactly generated space, then there exist a Banach space $Y$ and an isometry
$f:X\rightarrow Y$, but $X$ is not linearly isomorphic any subspace of $Y$.

\end{theorem}

Recently, Cheng, Dai, Dong and Zhou \cite{cheng1} showed that for Banach spaces $X$ and $Y$, if such an $\eps$-isometry
$f: X\rightarrow Y$ exists, then there is a linear isometry $U:X^{**}\rightarrow Y^{**}$; in particular, if $Y$ is reflexive,
then there is a linear isometry $U:X\rightarrow Y$.

{\bf $\eps$-isometry and stability.} In 1945,  Hyers and Ulam  proposed the following question\cite{hy} (see, also \cite{om}): whether
for every surjective $\eps$-isometry $f:X\rightarrow Y$ with
$f(0)=0$ there exist a surjective linear isometry $U:X\rightarrow Y$
and $\gamma>0$ such that
$$(1.1)\;\;\;\;\;\;\;\;\;\;\;\;\;\;\;\;\;\;\;\;\;\|f(x)-Ux\|\leq\gamma\varepsilon,~~~~{\rm for~~all}\;
 x\in X .\;\;\;\;\;\;\;\;\;\;\;\;\;\;\;\;\;\;\;\;\;\;\;\;\;\;\;\;\;\;\;\;\;\;$$

 After many years efforts of a number of mathematicians (see, for instance, \cite{ge}, \cite{gru}, \cite{hy}, and  \cite{om}),
 the following sharp estimate was finally obtained by Omladi\v{c} and \v{S}emrl \cite{om}.
\begin{theorem} [Omladi\v{c}-\v{S}emrl]
If $f:X\rightarrow Y$ is a surjective $\varepsilon$-isometry  with $f(0)=0$, then there is a surjective linear
isometry $U:X\rightarrow Y$ such that
$$ \|f(x)-Ux\|\leq2\varepsilon,~~~~{\rm for~~all}\;~~ x\in X. $$
\end{theorem}

 The study of nonsurjective  $\eps$-isometry has also brought to mathematicians' attention (see, for instance \cite{dil},
 \cite{om}, \cite{qian},  \cite{sm} and \cite{ta}).
Qian\cite{qian} first proposed  the following problem in 1995.
\begin{problem} Whether there exists a constant $\gamma>0$ depending only on
$X$ and $Y$ with the following property: For each standard
$\varepsilon$-isometry $f:X\rightarrow Y$ there is a
bounded linear operator $T:L(f)\rightarrow X$ such
that
$$(1.2) \;\;\;\;\;\;\;\;\;\;\;\; \;\;\; \;\;\;\;\;\;\;\;\;\;\|Tf(x)-x\|\leq \gamma\varepsilon,\;
\;{\rm for\;all}\;~x\in X.\;\;\;\;\;\;\;\;\;\;\;\;\;\;\;\;\;\;\;\;\;\;\;\;\;\;\;\;\;\;\;\;\;\;$$
\end{problem}

Then he showed that the answer is affirmative if both $X$ and $Y$ are
$L_p$ spaces. \v{S}emrl and V\"{a}is\"{a}l\"{a} \cite{sm} further presented a sharp estimate of (1.2)
with $\gamma=2$ if both $X$ and $Y$ are $L^p$ spaces for $1<p<\infty$.

However,  Qian (in the same paper \cite{qian}) presented the following simple counterexample.
\begin{example}[Qian]
Given $\varepsilon>0$, and let $Y$ be a separable Banach space admitting a
uncomplemented  closed subspace $X$. Assume that
$g$ is a bijective mapping from $X$ onto the closed unit ball $B_Y$ of $Y$ with $g(0)=0$. We define a
map $f:X\rightarrow Y$ by $f(x)=x+\varepsilon g(x)/2$ for all
$x\in X$. Then $f$ is an $\varepsilon$-isometry with $f(0)=0$ and
$L(f)=Y$. But there are no such $T$ and $\gamma$ satisfying (1.2).
\end{example}

This disappointment makes us to search for (1) some weaker stability version and (2) some appropriate complementability
assumption on some subspaces of $Y$ associated with the mapping. Cheng, Dong and Zhang \cite{cheng} showed the following
theorems about the two questions.

  \begin{theorem}[Cheng-Dong-Zhang]
 Let $X$ and $Y$ be Banach spaces,
  and let $f:X\rightarrow Y$ be a standard $\eps$-isometry for
 some $\eps\geq 0$. Then for every $x^*\in X^*$, there exists $\phi\in Y^*$ with $\|\phi\|=\|x^*\|\equiv r$ such that
 $$(1.3)\;\;\;\;\;\;\;\;\;\;\;\;\;\;\;\;\;\;|\langle\phi,f(x)\rangle-\langle x^*,x\rangle|\leq4\eps r, \;{\rm for\;all}
 \;x\in X.\;\;\;\;\;\;\;\;\;\;\;\;\;\;\;\;\;\;\;\;\;\;\;\;\;\;\;\;\;\;$$
 \end{theorem}

 For a standard $\eps$-isometry $f: X\rightarrow Y$, let  $$Y\supset E={\rm the\; annihilator\;of\; all\; bounded\;linear\;
 functionals\;and\;bounded\;on}$$$$ C(f)\equiv\overline{\rm co}(f(X),-f(X)),$$ i.e.
 $$(1.4)\;\;\;\;\;\;\;\;\;\;E=\{y\in Y:\langle y^*,y\rangle=0, y^*\in Y^*\;{\rm is\;bounded\;on}\;C(f)\}.\;\;\;\;\;\;\;\;\;\;\;\;\;\;\;\;$$

 \begin{theorem}[Cheng-Dong-Zhang]
 Let $X$ and $Y$ be Banach spaces,
  and let $f:X\rightarrow Y$ be a standard $\eps$-isometry for
 some $\eps\geq 0$. Then

 (i) If $Y$ is reflexive and if $E$ is $\alpha$-complemented in $Y$, then there is a bounded linear operator $T:Y\rightarrow X$ with
 $\|T\|\leq\alpha$ such that
$$\;\;\;\;\;\;\;\;\;\;\;\;\;\;\;\;\;\;\;\;\;\;\;\;\;\;\;\|Tf(x)-x\|\leq4\eps, \;{\text{ for all }}x\in X.
\;\;\;\;\;\;\;\;\;\;\;\;\;\;\;\;\;\;\;\;\;\;\;\;\;\;\;\;\;\;\;\;\;\;\;\;\;\;\;\;$$

(ii) If $Y$ is reflexive, smooth and locally uniformly convex, and if $E$ is $\alpha$-complemented in $Y$, then
there is a bounded linear operator $T:Y\rightarrow X$ with $\|T\|\leq\alpha$ such that the following sharp estimate holds
$$\;\;\;\;\;\;\;\;\;\;\;\;\;\;\;\;\;\;\;\;\;\;\;\;\;\;\;\|Tf(x)-x\|\leq2\eps, \;{\text{ for all }}x\in X.\;\;\;\;\;\;\;\;\;
\;\;\;\;\;\;\;\;\;\;\;\;\;\;\;\;\;\;\;\;\;\;\;\;\;\;\;\;\;\;\;$$

 \end{theorem}

{\bf Universal stability spaces of $\eps$-isometries.}
For study of the stability of $\eps$-isometry of Banach spaces, the following two questions are very natural.

\begin{problem}
Is there a  characterization for the class of Banach spaces $X$
satisfying that for every Banach space $Y$ there is $\infty>\gamma>0$ such that for every $\eps$-isometry
$f:X\rightarrow Y$ with $f(0)=0$,there exists a bounded linear
operator $T:$ $L(f)\rightarrow X$  such that
$$\;\;\;\;\;\;
\|Tf(x)-x\|\leq \gamma\varepsilon ,\;\;for\;all \;x\in X,$$ that is, inequality (1.2) holds.
Every space $X$ of this class is said to be a universal left-stability space.
\end{problem}

\begin{problem}
 Can we characterize the class of Banach spaces $Y$, such that for every Banach space $X$ there is $\infty>\gamma>0$ such that for every $\eps$-isometry $f:X\rightarrow Y$ with $f(0)=0$,
there exists a bounded linear operator $T:$
$L(f)\rightarrow X$  such that
$$\;\;\;\;\;\;
\|Tf(x)-x\|\leq \gamma\varepsilon ,\;\;for\;all \;x\in X.$$
 Every space $Y$ of this class is called a universal right-stability space.
\end{problem}

Recently, Cheng, Dai, Dong and Zhou \cite{cheng1}  studied
properties of  universal left-stability  and right-stability spaces. As a
result, they proved that up to linear isomorphism, universal right-stability
spaces are just Hilbert spaces; Every injective space is universally
left-stable, and a Banach space $X$ which is linear
isomorphic to a subspace of $\ell_\infty$ is universally left-stable
if and only if it is linearly isomorphic to $\ell_\infty$. They also
verified that a separable space $X$ satisfies that the pair $(X,Y)$
is stable for every separable $Y$ if and only if $X$ is linearly
isomorphic to $c_0$.

In this paper, we further show that if a dual Banach space $X$ is
universally left-stable, then it is isometric to a complemented
$w^*$-closed subspace of $\ell_\infty(\Gamma)$ for some set
$\Gamma$, hence, an injective space; and that a Banach space is
universally left-stable
if and only if it is a cardinality injective space. Therefore, the universal left-stability of
Banach spaces is invariant under linear isomorphism.\\

All symbols and notations in this paper are standard. We use $X$ to
denote a real Banach space and $X^*$ its dual. $B_X$ and $S_X$
denote the closed unit ball and the unit sphere of $X$,
respectively. For a subspace $E\subset X$, $E^\bot$ denotes the
annihilator of $E$, i.e. $E^\bot=\{x^*\in X^*:\langle
x^*,e\rangle=0\;{\rm for\;all}\;e\in E\}$. Given a bounded linear
operator $T:X\rightarrow Y$, $T^*:Y^*\rightarrow X^*$ stands for its
conjugate operator. For a subset $A\subset X \;(X^*)$,
$\overline{A}$,  ($w^*\text{-}\overline{A}$) and  $\text{co}(A)$
stand for the closure (the $w^*$-closure), and the convex hull of
$A$, respectively.

\section{ Universal left-stability dual spaces for $\eps$-isometries}

In this section, we search for some properties of the class of universal left-stability spaces for $\eps$-isometries.
After discussion of cardinality of Banach spaces, we show that a dual Banach space is universally left-stable if and only if
it is an injective space.

Recall that a Banach space $X (Y)$ is universally left(right)-stable if it satisfies that for every Banach space $Y (X)$ there is $\infty>\gamma>0$
such that for every standard $\eps$-isometry $f:X\rightarrow Y$,
there exists a bounded linear operator $T:$
$L(f)\rightarrow X$ such that
$$(2.1)\;\;\;\;\;\;\;\;\;\;\;\;\;\;\;\;\;\;\;\;\;\;\;\;
\|Tf(x)-x\|\leq \gamma\varepsilon ,\;\;for\;all \;x\in
X.\;\;\;\;\;\;\;\;\;\;\;\;\;\;\;\;\;\;\;\;\;\;\;\;\;\;\;\;\;\;\;\;\;\;\;\;$$
As a result we show that (1) inequality (2.1) holds for every Banach
space $Y$ if and only if $X$ is a cardinality injective Banach
space; (2) if a dual Banach space $X$ is universally left-stable,
then it is isometric to a complemented $w^*$-closed subspace of
$\ell_\infty(\Gamma)$ for some set $\Gamma$, hence, an injective
space.

The following lemma is presented in \cite{cheng1}.

\begin{lemma}
Let $X$ be a closed subspace of a Banach space $Y$. If $\text{card}(X)=\text{card}(Y)$, then  for every $\eps>0$ there is a standard $\eps$-isometry $f:X\rightarrow Y$
such that

(1)  $L(f)\equiv\overline{\text{span}}f(X)=Y$;

(2)  $X$ is complemented whenever $f$ is stable.
\end{lemma}

A Banach space $X$ is said to be  injective if it has the following
extension property: Every bounded linear operator from a closed
subspace of a Banach space into $X$ can be extended to be a bounded
operator on the whole space. $X$ is called isometrically injective
if every such bounded operator has a norm-preserved extension (See,
for example \cite{Alb}). Goondner \cite{Goo} introduced a family of
Banach spaces coinciding with the family of injective spaces: for
any $\lambda\geq1$, a Banach space X is a $P_\lambda$-space if,
whenever $X$ is isometrically embedded in another Banach space,
there is a projection onto the image of X with norm not larger than
$\lambda$. The following result was due to Day \cite{day} (see,
also,  Wolfe \cite{w}, Fabian et al. \cite{fa}, p. 242).

\begin{proposition} A Banach space $X$ is (isometrically) injective
if and only if it is a $P_\lambda$ ($P_1$)-space for some $\lambda\geq1$.
\end{proposition}

Goondner \cite{Goo}, Nachbin \cite{Nac} and Kelley \cite{Kel} characterized the isometrically injective spaces.

\begin{theorem}[ Goodner-Kelley-Nachbin, 1949-1952] A Banach space
is isometrically injective if and only if it is isometrically
isomorphic to the space of continuous functions $C(K)$ on an
extremely disconnected compact Hausdorff space $K$, i.e. the space
$K$ such that the closure of any open set is open in $K$.

\end{theorem}

\begin{remark}
For any set $\Gamma$ endowed with the discrete metric topology,
$\ell_\infty(\Gamma)$ is isomorphic to
$C(K_\Gamma)$, where
$K_\Gamma$ is the Stone-$\check{C}$ech compactification of $\Gamma$.
By Goodner-Kelley-Nachbin's theorem,  $\ell_\infty(\Gamma)$ is an
isometrically injective space.
\end{remark}

\begin{lemma}
Suppose that $X$ is a Banach space. Then
$$(2.2)\;\;\;\;\;\;\;\;\;\;\;\;\;\;\;\text{dens}(X)\geq w^*\text{-dens}(B_{X^*})\geq w^*\text{-dens}(X^*).\;\;\;\;\;\;\;\;\;\;\;\;\;\;\;\;\;\;\;\;\;\;\;\;\;\;\;\;\;\;$$

\end{lemma}

\begin{proof}
It is trivial if $\dim X<\infty$. Assume that $\dim X=\infty$.
  Note dens$X$=dens$S_X$. Let $(x_\alpha)\subset S_{X}$ be a dense subset with card$(x_\alpha)=$des$S_X$, and let $\phi$ be a selection of the subdifferential mapping $\partial\|\cdot\|: X\rightarrow 2^{S_{X^*}}$ of the norm $\|\cdot\|$ defined by $$\partial\|x\|=\{x^*\in X^*: \|x+y\|-\|x\|\geq\langle x^*,y\rangle,\;\;{\rm for\;all\;}y\in X\}.$$ Then
  $$(2.3)\;\;\;\;\;\;\;\;\;\;\;\;\;\;\;\;\;\;\;\;\;\;\;\;\;\;\;\;\;\;w^*{\text-}\overline{{\rm co}}(\phi(x_\alpha))=B_{X^*}.\;\;\;\;\;\;\;\;\;\;\;\;\;\;\;\;\;\;\;\;\;\;\;\;\;\;\;\;\;\;\;\;\;\;\;\;\;\;\;\;\;\;\;\;\;\;\;\;\;\;$$
  Since $\phi(x_\alpha)$ is a norming set of $X$, i.e. $$\|x\|=\sup_\alpha\langle\phi(x_\alpha),x\rangle,\;\text{ for all }x\in X,$$ and since $\dim X=\infty$, $card(\phi(x_\alpha))=\infty.$ Hence,
  $$\text{dens}(X)=\text{dens}( S_X)
  =\text{card}(x_\alpha)\geq\text{card}(\phi(x_\alpha))
  =\text{dens}\{co(\phi(x_\alpha))\}$$
  $$\geq w^*\text{-dens}\{w^*\text{-}\overline {co}(\phi(x_\alpha))\}=w^*\text{-dens}(B_{X^*})\geq w^*\text{-dens}({X^*}).$$
  Equality (2.3) is used to the last equality above.

\end{proof}

We should mention here that the inequalities in Lemma 2.5 can be
proper. For example, let $X=\ell_\infty$. Then $dens(X)=card[0,1]$,
but $B_{X^*}$ is $w^*$-separable. On the other hand, we put an
equivalent norm $|\|\cdot|\|$ on $\ell_\infty$ by
$$|\|x\||=\frac{1}2(\|x\|_\infty+\limsup_n|x(n)|),\;{\rm
for\;all}\;x=(x(n))\in\ell_\infty.$$ Then
  $X^*=\ell_1\oplus c_0^\bot$  is $w^*$-separable, but its closed unit ball $B_{X^*}$ is not $w^*$-separable.

\begin{lemma}
Let $X$ be a Banach space with $\dim X\geq1$, and let $\Omega=dens(X)$. Then $$card(X)=card(c_0(\Omega)).$$

\end{lemma}

\begin{proof}
If $X$ is separable, then $dens(X)=\aleph_0$ and $$card(X)=\Omega^{\aleph_0}={\aleph_0}^{\aleph_0}=
\aleph^{\aleph_0}=card(c_0({\aleph_0}))=card(c_0(\Omega)).$$
If $X$ is not separable, then $$\aleph\leq\Omega=\Omega^{\aleph_0}=card(X).$$ Let $e_\omega$
(for all $\omega\in\Omega$)  be the standard unit vectors in $c_0(\Omega)$.
$$card(c_0(\Omega))=dens(c_0(\Omega))^{\aleph_0}=dens(c_0(\Omega))$$
$$=\text{dens(span}(e_\omega)_{\omega\in\Omega})=\sum_{n=1}^\infty(\aleph_0\cdot\Omega)^n=\Omega.$$

\end{proof}

\begin{theorem}
Suppose that  $X$ is a universal left-stability  spaces. Then
there is an injective conjugate space $V$ such that $X\subset
V\subset X^{**}$.
\end{theorem}

\begin{proof}
Let $(x^*_\gamma)\subset B_{X^*}$ be a $w^*$-dense subset of
$B_{X^{*}}$ with ${\rm card}(x^*_\gamma)=w^*\text{-dens}
B_{X^*}\equiv\Gamma$, where $\Gamma$ denotes the cardinality of the
$w^*$-density of $B_{X^*}$.   By Lemma 2.5, ${\rm
dens}(X)\geq\Gamma$. Let $\Phi: X\rightarrow\ell_\infty(\Gamma)$ be
defined by
$$\Phi(x)=(\langle x^*_\gamma,x\rangle)_{\gamma\in\Gamma},\;\;x\in X.$$
Clearly, $\Phi$ is a linear isometry. Now, there are two
$w^*$-topologies on $X$: One is the $w^*$-topology of $X^{**}$
restricted to $X\subset X^{**}$, which we denote by $\tau_{w^*}$;
the other is, acting as a subspace of $\ell_\infty(\Gamma)$, the
$w^*$-topology of $\ell_\infty(\Gamma)=\ell_1(\Gamma)^*$ restricted
to $\Phi(X)$, which we denote by $\tau_{w^*,\infty}$. Let
$V=\tau_{w^*,\infty}\text{-}\overline{\Phi(X)}$, the
$\tau_{w^*,\infty}$-closure of $\Phi(X)$ in $\ell_\infty(\Gamma)$.
Without loss of generality, we assume $\Phi(X)=X$. Note $X^{**}=$
the $w^*$-closure of $X$ in
$\ell_\infty(\Gamma)^{**}=\ell_\infty\bigoplus
(c_0(\Gamma)^\bot)^*$. It is easy to get $X\subset V\subset X^{**}$.

Let $Z=\overline{\text{span}}(X\cup c_0(\Gamma))$. By Lemma 2.5 and
Lemma 2.6, card$(X)=$card$(Z).$

Let $g: X\rightarrow B_Z$ be a bijection with $g(0)=0$, and let $f: X\rightarrow Z$ be defined by
$$f(x)=x+\eps/2\cdot g(x),\;\;{\rm for\;all}\;x\in X.$$ Clearly, $f$ is a standard $\eps$-isometry.
Note $\lim_{n\rightarrow\infty}f(nx)/n=x,\;\;{\rm for\;all\;}x\in X.$ We see that $L(f)=Z$.
By universal left-stability assumption, there exist a bounded linear operator $T:Z\rightarrow X $ and a positive number $\gamma>0$ such that
$$\|Tf(x)-x\|\leq\gamma\eps,\;\;{\rm for\;all\;}x\in X.$$
It is easy to see that $T$ is surjective. Thus its biconjugate
operator $T^{**}:Z^{**}\rightarrow X^{**}$ is surjective and
$w^*$-to-$w^*$ continuous with $\|T^{**}\|=\|T\|$ and with
$T^{**}|_Z=T$. Since $Z$ isometrically contains $c_0(\Gamma)$,
$Z^{**}$ (isometrically) contains $\ell_\infty(\Gamma)$. Note that
$Z$ acting as a subspace of $\ell_\infty(\Gamma)$, its $w^*$-closure
in $\ell_\infty(\Gamma)^{**}=\ell_\infty(\Gamma)\bigoplus
(c_0(\Gamma)^\bot)^*$ is just $Z^{**}$. Let
$P:\ell_\infty(\Gamma)^{**}\rightarrow\ell_\infty(\Gamma)$ be the
natural projection. Clearly, $P$ is $w^*$-to-$w^*$ continuous with
$\|P\|=1$.  On the other hand, since $Z$ contains $c_0(\Gamma)$, it
is a $w^*$-dense subspace of $\ell_\infty(\Gamma)$. Therefore, its
$w^*$-closure in $\ell_\infty(\Gamma)=\ell_1(\Gamma)^*$ is just the
whole space $\ell_\infty(\Gamma)$.
 We claim that $PT^{**}\ell_\infty(\Gamma)=V.$

Let $S=PT^{**}$. Then $S: Z^{**}\rightarrow\ell_\infty(\Gamma)$ is
$w^*$-to-$w^*$ continuous. It suffices to show
$S(\ell_\infty(\Gamma))\subset V.$
 Given $z\in\ell_\infty(\Gamma)$,
since $B_Z$($B_{c_0(\Gamma)}$, resp.) is $w^*$-dense in
$B_{Z^{**}}$($B_{\ell_\infty(\Gamma)}$,resp.), there is a net
$(z_\alpha)\subset\|z\|B_Z$ such that $z_\alpha\rightarrow z$ in
both the $w^*$ topologies of $\ell_\infty(\Gamma)$ and
$\ell_\infty(\Gamma)^{**}.$ Thus, $S(z_\alpha)\rightarrow
S(z)\in\ell_\infty(\Gamma),$ in the $w^*$-topology of
$\ell_\infty(\Gamma)$. Since
$S(z_\alpha)=PT^{**}(z_\alpha)=PT(z_\alpha)=T(z_\alpha)\in X\subset
V$, and since $V$ is a $w^*$-closed subspace of
$\ell_\infty(\Gamma),$ we see $S(z)\in V$.

 We have proven that $V$ is a complemented $w^*$-closed subspace of $\ell_\infty(\Gamma)$,
 which in turn entails that $V$ is a conjugate injective space.
\end{proof}

\begin{corollary}
Suppose that  $X$ is a universally left-stable conjugate Banach spaces.     Then it is an injective space.
\end{corollary}

\begin{proof}
Suppose that $W$ is a Banach space with
$W^*=X$. Let $\Gamma$ be the cardinality of the $w^*$-density of
$B_{X^*}=B_{W^{**}}$ and let $\Omega=\text{dens}(W)=\text{dens}(B_W).$  Then, by the separation theorem of convex sets, 
$\Omega=\Gamma$. Indeed, on one hand, any dense subset of $B_W$ is necessarily $w^*$-dense in $B_{W^{**}}$; on the other hand, for any subset $S$  of $B_W$ with cardinality less than $\Omega$ we see that $U\equiv\overline{{\rm span}}S$ is a closed proper subspace of $W$. Therefore, by the separation theorem, there exists non-zero functional $\phi\in W^*=X$ such that $\phi(U)=\{0\}$. Hence, $U$ is not $w^*$-dense in $W^{**}=X^*$, which entails that $S\subset B_W$ is not $w^*$-dense in $B_{X^*}$.  Assume that
$(w_\gamma)_{\gamma\in\Gamma}\subset B_{W}$ is a dense subset of
$B_{W}$. Let $\Phi:X\rightarrow\ell_\infty(\Gamma)$ be defined by
$$\Phi(x)=(\langle w_\gamma,x\rangle)_{\gamma\in\Gamma}.$$  Clearly,
$\Phi$ is a $w^*$-to-$w^*$
continuous linear isometry. Conversely, since $(w_\gamma)$ is dense in $B_W$,
$\Phi^{-1}$ is also $w^*$-to-$w^*$ continuous. Thus, $\Phi(X)$ is a
$w^*$-closed space of $\ell_\infty(\Gamma)$. According to  definition of the space $V$ in Theorem 2.7, $\Phi(X)=V$. 
So that $\Phi(X)$ is an injective space, hence, $X$ is also injective.

\end{proof}

\section{A characterization of universal-left-stability spaces}

\begin{definition}

A Banach space $X$ is said to be cardinality injective, if there exists a constant $\lambda\geq 0$ such that for every
Banach space $Y$ isometrically containing $X$ and  with the same
cardinality as $X$, i.e., $\text{card}(Z)=\text{card}(X)$,
 there is a projection $P:Y\rightarrow X$ with $\|P\|\leq\lambda$.
 \end{definition}
Using the same procedure of Day \cite{day}, we can show that $X$ is  cardinality injective if and only if it has the following extension property:
Every bounded linear operator from a closed subspace of a Banach space $Y$ with $\text{card}(Y)\leq\text{card}(X)$ into $X$ can be extended to be a bounded operator on the whole space. Thus, we have the following property.
\begin{proposition}
A Banach space isomorphic to a cardinality injective space is again a cardinality injective space.
\end{proposition}

The following theorem says that a Banach space is universally left-stable if and only if it is cardinally injective.

\begin{theorem} Let $X$ be a Banach space. Then a sufficient and necessary condition
for that there is $\gamma>0$ such that for every Banach space $Y$,
every nonnegative number $\eps$ and for every standard
$\eps$-isometry $f:X\rightarrow Y$ there exists  a bounded linear
operator $T: L(f)\rightarrow X$  satisfying
 $$\|Tf(x)-x\|\leq\gamma\eps, \;\;{\rm for\;all}\; x\in X$$ is that $X$ is a cardinality injective space.

\end{theorem}

\begin{proof}  Sufficiency. Assume that $X$ is a cardinality injective Banach space.
Then there exists $\alpha>0$ such that for every Banach space $Z$
isometrically containing $X$ there is a projection $P:Z\rightarrow
X$ such that $\|P\|\leq\alpha$. We can assume that $X$ is a closed
subspace of $\ell_\infty(\Gamma)$ for some set $\Gamma$; otherwise,
we can identify $X$ for $J_X(X)$ as a closed subspace of
$\ell_\infty(\Gamma)$, where $\Gamma$ denotes the closed ball
$B_{X^*}$ of $X^*$. Given any $\beta\in \Gamma$, let $\delta_\beta
\in\ell_\infty(\Gamma)^*$ be defined for $x=(x(\gamma))_{\gamma \in
\Gamma}\in \ell_\infty(\Gamma)$ by $\delta_\beta (x)=x(\beta).$

Assume that $f : X\rightarrow Y$ be an $\eps$-isometry with $f(0)=0$.
For every $x^* \in X^*$, 
by Theorem 1.7, there is
$\phi\in Y^*$ with $\|\phi\|=\|x^*\|$ such that
 $$(3.2)\;\;\;\;\;\;\;\;\;\;\;\;\;\;\;\;|\langle\phi,f(x)\rangle-\langle x^*,x\rangle|
\leq4\eps\|x^*\|,\; {\rm for \;all}\; x\in
X.\;\;\;\;\;\;\;\;\;\;\;\;\;\;\;\;\;\;\;\;\;\;\;\;\;\;\;\;$$ In
particular, letting $x^*=\delta_\gamma$ in (3.2) for every fixed
$\gamma\in\Gamma$, we obtain a linear functional $\phi_\gamma\in
Y^*$ satisfying (3.2) with
$\|\phi_\gamma\|=\|\delta_\gamma\|_X\leq1.$ Therefore,
$$S(y)=(\phi_\gamma(y))_{\gamma\in\Gamma},\;{\rm for\; every}\; y\in
Y$$ defines a linear operator $S:Y\rightarrow \ell_\infty(\Gamma)$
with $\|S\|\leq1$. Let $Z={\rm span}[S(f(X))\cup X]$. Then
$\overline{Z}\supset X$ and ${\rm card}(Z)={\rm card}(X)$. Since $X$
is a cardinality injective space, there is a projection $P:
\overline{Z}\rightarrow X$ with $\|P\|\leq\alpha<\infty.$

Let $T(y)=P(S(y))$, for all $y\in L(f)$, and note $P|_X=I_X$, the
identity from $X$ to itself. Then $\|T\|\leq\|P\|\|S\|\leq\alpha$
and for all $x\in X$,

$$\|Tf(x)-x\|=\|P(\phi_{\gamma}(f(x)))_{\gamma \in
\Gamma}-(\delta_{\gamma}(x))_{\gamma \in \Gamma}\|$$

$$\;\;\;\;\;\;\;\;\;\;\;=\|P(\phi_{\gamma}(f(x)))_{\gamma \in
\Gamma}-P((\delta_{\gamma}(x))_{\gamma \in \Gamma})\|\;\;\;\;\;\;\;\;\;\;\;\;\;$$

 $$\;\;\;\;\;\;\;\;\;\;\;\;\leq\|P\|\cdot\|(\phi_{\gamma}(f(x)))_{\gamma \in
\Gamma}-(\delta_\gamma(x))_{\gamma \in \Gamma}\|_\infty\leq 4 \alpha\eps.$$
We finish the proof of the sufficiency  by taking $\gamma=4\alpha$.

Necessity. Suppose, to the contrary, that $X$ is not a cardinality
injective space. Then, there is a Banach space $Y$ containing $X$
and a dense subspace $Z$ of $Y$ such that ${\rm card}(Z)={\rm
card}(X)$, but $X$ is not complemented in $Y$. By Lemma 2.1, there is a 
standard $\eps$-isometry $f:X\rightarrow Y$ which is not stable.

\end{proof}

\begin{corollary} The universally-left-stability of Banach spaces is invariant under linear
isomorphism.
\end{corollary}

\begin{proof}
It suffices to note Theorem 3.3 and Proposition 3.2.
\end{proof}

\begin{corollary}
Let $X$ be a  Banach space. If  for every Banach space $Y$, every
nonnegative number $\eps$ and for every standard $\eps$-isometry
$f:X\rightarrow Y$ there exist $\infty>\gamma\geq0$ and a bounded
linear operator $T: L(f)\rightarrow X$  satisfying
 $$\|Tf(x)-x\|\leq\gamma\eps, \;\;{\rm for\;all}\; x\in X,$$ then $X$ is universally left-stable,
i.e. the positive number $\gamma$ can be chosen depending only on
$X$.
\end{corollary}

\begin{proof}
By the Theorem 3.3, it suffices to show that $X$ is a $\lambda$
-cardinally injective. But it is the same proof of the necessity of
the theorem above.
\end{proof}

\begin{remark}
Cheng, Dai, Dong and Zhou \cite{cheng1} recently showed that every injective Banach space is universally left-stable. By this result and Corollary 2.8, we get that a dual Banach space is injective if and only if it is universally left-stable. This and Theorem 3.3 entail that a dual Banach space is cardinality injective if and only if it is injective. But we do not know whether it is true in general.

\end{remark}

\bibliographystyle{amsalpha}

\end{document}